\documentclass[12pt,a4paper,reqno]{article}
 \usepackage{amsthm, mathrsfs,amssymb,amsmath,}
 \usepackage{enumerate}
 \usepackage{mathtools}
 \usepackage{authblk}
 
 \usepackage{hyperref}
 \allowdisplaybreaks
 \numberwithin{equation}{section}
 
 \usepackage{graphicx}
 \usepackage{enumitem}

 \usepackage{xcolor}

 \theoremstyle{plain}
 
 \newtheorem{theorem}[subsection]{Theorem}

 \newtheorem{lemma}[subsection]{Lemma}
 \newtheorem{defi}[subsection]{Definition}
 \newtheorem{prop}[subsection]{Proposition}

 \theoremstyle{definition}

 \newtheorem{remark}[subsection]{Remark}

 \newcommand{\tarc}{\mbox{\large$\frown$}}
 \newcommand{\arc}[1]{\stackrel{\tarc}{#1}}

 \newcommand{\jj}{\vee}
 \newcommand{\mm}{\wedge}
 \newcommand{\JJ}{\bigvee}
 \newcommand{\MM}{\bigwedge}
 \newcommand{\JJm}[2]{\JJ(\,#1\mid#2\,)}
 \newcommand{\MMm}[2]{\MM(\,#1\mid#2\,)}

 \newcommand{\uu}{\cup}
 \newcommand{\ii}{\cap}
 \newcommand{\UU}{\bigcup}
 \newcommand{\II}{\bigcap}
 \newcommand{\UUm}[2]{\UU\{\,#1\mid#2\,\}}
 \newcommand{\IIm}[2]{\II\{\,#1\mid#2\,\}}
 
 \newcommand{\ci}{\subseteq}
 \newcommand{\nc}{\nsubseteq}
 \newcommand{\sci}{\subset}
 \newcommand{\nci}{\nc}
 \newcommand{\ce}{\supseteq}
 \newcommand{\sce}{\supset}
 \newcommand{\nce}{\nsupseteq}
 \newcommand{\nin}{\notin}
 \newcommand{\es}{\emptyset}
 \newcommand{\set}[1]{\{#1\}}
 \newcommand{\setm}[2]{\{\,#1\mid#2\,\}}
 
 \newcommand{\nle}{\nleq}
 
 \newcommand{\ga}{\alpha}
 \newcommand{\gb}{\beta}
 \newcommand{\gc}{\chi}
 \newcommand{\gd}{\delta}
 \renewcommand{\gg}{\gamma}
 \newcommand{\gh}{\eta}
 \newcommand{\gi}{\iota}
 \newcommand{\gk}{\kappa}
 \newcommand{\gl}{\lambda}
 \newcommand{\gm}{\mu}
 \newcommand{\gn}{\nu}
 \newcommand{\go}{\omega}
 \newcommand{\gp}{\pi}
 \newcommand{\gq}{\theta}
 \newcommand{\gs}{\sigma}
 \newcommand{\gt}{\tau}
 \newcommand{\gx}{\xi}
 \newcommand{\gy}{\psi}
 \newcommand{\gz}{\zeta}
 \newcommand{\vp}{\varphi}
 
 \newcommand{\gG}{\Gamma}
 \newcommand{\gD}{\Delta}
 \newcommand{\gF}{\Phi}
 \newcommand{\gL}{\Lambda}
 \newcommand{\gO}{\Omega}
 \newcommand{\gP}{\Pi}
 \newcommand{\gQ}{\Theta}
 \newcommand{\gS}{\Sigma}
 \newcommand{\gX}{\Xi}
 \newcommand{\gY}{\Psi}
 
 \newcommand{\tbf}{\textbf}
 \newcommand{\tit}{\textit}
 
 \newcommand{\mbf}{\mathbf}
 \newcommand{\B}{\boldsymbol}
 \newcommand{\C}[1]{\mathcal{#1}}
 \newcommand{\D}[1]{\mathbb{#1}}
 \newcommand{\F}[1]{\mathfrak{#1}}
 
 \newcommand{\te}{\text}
 \newcommand{\tei}{\textit}
 \newcommand{\im}{\implies}
 \newcommand{\ti}{\times}
 \newcommand{\exi}{\ \exists \ }
 \newcommand{\ther}{\therefore}
 \newcommand{\be}{\because}
 \newcommand{\ep}{\epsilon}
 \newcommand{\la}{\langle}
 \newcommand{\ra}{\rangle}
 \newcommand{\ol}{\overline}
 \newcommand{\ul}{\underline}
 \newcommand{\q}{\quad}
 \newcommand{\qq}{\qquad}
 \newcommand{\fa}{\ \forall \, }
 \newcommand{\nd}{\noindent}
 \newcommand{\bs}{\backslash}
 \newcommand{\pa}{\partial}
 \newcommand{\para}{\parallel}
 \newcommand{\sm}{\setminus}
 \newcommand{\tl}{\tilde}
 
 \usepackage{fancyhdr}
 \title{Quantization dimensions for the bi-Lipschitz recurrent Iterated function systems}
 \usepackage{authblk}

 \author[1]{Amit Priyadarshi}
 \author[2]{Mrinal K. Roychowdhury}
 \author[3*]{Manuj Verma}
 
 \affil[1,3*]{Department of Mathematics, Indian Institute of Technology Delhi, New Delhi, India 110016}
 \affil[2] {School of Mathematical and Statistical Sciences, University of Texas Rio Grande Valley, 1201
 	West University Drive, Edinburg, TX 78539-2999, USA}
 
 \affil[1]{priyadarshi@maths.iitd.ac.in}
 \affil[2]{mrinal.roychowdhury@utrgv.edu}
 \affil[3*]{correspondence to - mathmanuj@gmail.com}

\begin{document}

	\date{}	
Published, Dynamical Systems 	
	

	
	\maketitle
	
	

	
	
%
%
%
%
	
	


	

	\begin{abstract}
		In this paper,  the quantization dimensions of the Borel probability measures supported on the limit sets of the bi-Lipschitz recurrent  iterated function systems under the strong  open set condition in terms of the spectral radius have been estimated.
	\end{abstract}
	

	\section{Introduction}
	The term `quantization' was first used in signal processing and data compression theory, and quantization is a process of approximating a probability measure by a discrete probability measure supported on a finite set.
	In quantization theory, there are two main types of problems. The first is to determine the quantization dimension (see \cite{GL1}) of a probability measure, and second is to find out the optimal sets (see \cite{GL1}) of $n$-means for a probability measure. There is  some literature available on determining the $n$-optimal sets; see \cite{R7,R6}, for more details. In this paper, we focus on the estimation of the quantization dimension. A few decades ago, Zador (see \cite{ZADOR}) introduced the term `quantization dimension' and gave a formula for it. Quantization dimension is also connected with other dimensions of dynamical systems; for example, the lower quantization dimension of a probability measure lies between the Hausdorff dimension and the lower box counting dimension of the measure, and the upper quantization dimension lies between the packing dimension and the upper box counting dimension of the measure (see  \cite{DAi,Potz}). Thus, the quantization dimension is introduced as a new type of the fractal dimension spectrum. There are various techniques to determine the fractal dimension of sets and the graphs of continuous functions such as operator theoretic method \cite{Ban,Nussbaum1,Pr1}, function spaces technique \cite{SS5,Verma21}, Fourier transform method \cite{Fal}, mass distribution technique \cite{MF2,Fal,JH}, oscillation technique \cite{Fal,w&l}, potential theoretic method \cite{Fal}, etc. The fractal dimension for the fractional integrals of some continuous functions can be found in \cite{SS51,Liang}.  The basic results on the quantization dimension and its relationship to other fractal dimensions are also provided in the Graf-Luschgy's book (see \cite{GL1}).
	\par
	Given a Borel probability measure $\mu $ on $\mathbb{R}^k$, a number $r \in (0, \infty)$ and $ n \in \mathbb{N}$, the \tit{$n$th quantization error} of order $r$ for $\mu $ is defined by $$V_{n,r}(\mu):=\inf \Big\{\int \rho(x, A)^r d\mu(x): A \subset \mathbb{R}^k, \, \text{card}(A) \le n\Big\},$$
	where $\rho(x, A)$ represents the Euclidean distance of the point $x$ from the set $A$. If $\int\|x\|^rd\mu(x)<\infty$ and the support of $\mu$ contains infinitely many elements, then there exists a set $A$ with $\text{card}(A)=n$ such that $V_{n,r}(\mu)=\int \rho(x, A)^r d\mu(x)$ (see \cite{GL1}).  Any set $A$ for which the above infimum occurs is called an \tit{$n$-optimal set} of order $r.$ Let $e_{n,r}(\mu)=V_{n,r}^{\frac{1}{r}}(\mu)$. We define the \tit{quantization dimension} of order $r$ of $\mu $ by $$D_r=D_r(\mu):= \lim_{n \to \infty} \frac{ \log n}{- \log \big(e_{n,r}(\mu)\big)},$$
	provided the limit exists.
	If the limit does not exist, then we define the lower and the upper quantization dimensions by taking the limit inferior and the limit superior of the above sequence and are denoted by $\underline{D}_r$ and $\overline{D}_r$, respectively.
	\par Recall that, the system $\mathcal{J}=\{\mathbb{R}^k; f_1,f_2,\ldots, f_N\}$ is called an \tit{iterated function system (IFS)}, if each $f_i$ is a contraction mapping on $\mathbb{R}^k$ such that $\|f_i(x)-f_i(y)\|\leq s_i \|x-y\|~\forall~x,y \in \mathbb{R}^{k},$ where $0<s_i<1.$ By the Banach contraction theorem, there exists a unique nonempty compact set $E\subset\mathbb{R}^k$ such that $E=\cup_{i=1}^{N}f_i(E)$. Furthermore, if $(p_1,p_2,\ldots,p_N)$ is a probability vector corresponding to the IFS $\mathcal{J}$, then there exists a unique Borel probability measure $\mu$ supported on the set $E$ such that $\mu=\sum_{i=1}^{N}p_i \mu\circ f_{i}^{-1}.$ We call the set $E$ as the \tit{attractor} and $\mu$ as the \tit{invariant measure} of the IFS $\mathcal{J}.$
	We say that the IFS $\mathcal{J}$ satisfies the \tit{strong separation condition (SSC)} if $f_i(E)\cap f_j(E)=\emptyset$ for all $i\ne j$, and the \tit{open set condition (OSC)} if there exists a nonempty bounded open set $U\subset \mathbb{R}^k$ such that $f_i(U)\subseteq U$, $f_i(U)\cap f_j(U)=\emptyset$ for all $1\leq i\ne j\leq N.$ We say that the IFS $\mathcal{J}$ satisfies the \tit{strong open set condition (SOSC)} if the IFS $\mathcal{J}$ satisfies the open set condition with $U\cap E\ne \emptyset$ (see \cite{Fal,AS}). In \cite{Hoch}, Hochman introduced a weaker separation condition namely `exponential separation condition' and proved the dimension drop conjecture under this separation condition.
	\par
	Let $\{\mathbb{R}^k;f_1,f_2,\ldots,f_N\}$ be a self-similar IFS with probability vector $(p_1,p_2,\ldots,p_N)$ such that $\|f_i(x)-f_i(y)\|=s_i\|x-y\|$, where $0<s_i<1$ for each $i\in \{1,2,\ldots,N\}$. Let $E$ be the attractor and $\mu$ be the invariant measure corresponding to this self-similar IFS. Under the OSC, Graf and Luschgy (see \cite{GL2}) proved that the quantization dimension $D_r$ of this invariant measure $\mu$ is uniquely determined by the following formula
	\begin{align}\label{for}
		\sum_{i=1}^{N}(p_i s_{i}^{r})^{\frac{D_r}{r+D_r}}=1.
	\end{align}
	After that, Lindsay and Mauldin (see \cite{Lindsay}) generalized the Graf-Luschgy's result (see \cite{GL2}) to $F$-conformal measure associated with conformal IFS consisting of finitely many conformal mappings. From \cite{Pat}, we understand that the multifractal spectrum is the Legendre transform of the  temperature function (see \cite{Lindsay}). Also, the quantization dimension function is connected with the temperature function as can be seen from the Lindsay-Mauldin's paper \cite{Lindsay}. For more details on the multifractal analysis of measures and temperature function see \cite{Fal,ZLI,Pat,Bilal,Bilal2}. 
\par
	In \cite{R9,MK1}, Roychowdhury tried to extend  the Graf-Luschgy's result \cite{GL2} to the self-similar recurrent IFS and  bi-Lipschitz recurrent IFS, respectively. In these papers the author has used an incorrect invariance relation for the measure supported on the limit set of the recurrent IFS (see Remark \ref{re3.4}). Thus, the proofs in these papers are not correct. In this paper, we provide an invariance inequality (see Lemma \ref{le2.5}) for the measure supported on the limit set of the recurrent IFS  and using this, we determine the bounds of the quantization dimension of the Borel probability measure supported on the limit set of the bi-Lipschitz recurrent  IFS under the SOSC in terms of spectral radius. Note that the technique used in the current paper is also different from those in \cite{R9,MK1}. Moreover, this paper also generalizes the results of Graf and Luschgy \cite{GL1} in a more general setting.
	\par
	The paper is organized as follows. In the upcoming Section \ref{se2}, we provide some definitions and basic results. In Section \ref{se3}, we prove our main result.

	\section{Preliminaries}\label{se2}
	In this section, we give some basic results and definitions.
	\begin{defi}
		Let $A$ be a subset of $\mathbb{R}^k$. The Voronoi region of $a\in A$ is defined by
		$$W(a|A)=\{x\in X : \rho(x,a)=\rho(x,A)\}$$
		and the set $\{W(a|A) :  a\in A\}$ is called the Voronoi diagram of $A$.
	\end{defi}
	\begin{lemma}(see \cite{GL1})\label{le2.11}
		Let $\mu$ be a Borel probability measure on $\mathbb{R}^k$ with compact support $A_{*}$ and $r\in (0,\infty)$. Let $A_n$ be an $n$-optimal set for measure $\mu$ of order $r$. Define
		$$\|A_n\|_{\infty}=\max\limits_{a\in A_n}~~\max\limits_{x\in W(a|A_n)\cap A_{*}}\rho(x,a).$$
		Then $$\Big(\frac{\|A_n\|_{\infty}}{2}\Big)^r\min\limits_{x\in A_{*}}\mu\Big(B\Big(x,\frac{\|A_n\|_{\infty}}{2}\Big)\Big)\leq V_{n,r}(\mu).$$
	\end{lemma}
	Note that,
	if $\mu$ is a Borel probability measure on $\mathbb{R}^k$ with compact support, then for every $n\in \mathbb{N}$, there exists a finite set $A_n\subset \mathbb{R}^k$ such that
	$$V_{n,r}(\mu)=\int \rho(x,A_n)^rd\mu(x).$$
	This $A_n$ is an $n$-optimal set for measure $\mu$ of order $r$.

	\begin{lemma}(see \cite{GL1})\label{le2.9}
		Let $\mu$ be a Borel probability measure on $\mathbb{R}^k$ with compact support. Then, for any $r\in (0,\infty)$
		$$V_{n,r}(\mu)\to 0,~~\text{as}~~n\to \infty.$$
	\end{lemma}
	
	\begin{prop}(see \cite{GL1})\label{prop2.12}
		\begin{enumerate}
			\item If $0\leq t_1<\overline{D}_r<t_2,$ then
			$$\limsup_{n\to \infty}n e_{n,r}^{t_1}=\infty~~~\text{and}~~~\lim_{n\to \infty}n e_{n,r}^{t_2}=0.$$
			\item If $0\leq t_1<\underline{D}_r<t_2,$ then
			$$\liminf_{n\to \infty}n e_{n,r}^{t_2}=0~~~\text{and}~~~\lim_{n\to \infty}ne_{n,r}^{t_1}=\infty.$$
		\end{enumerate}
		
	\end{prop}
	Let $X\subset \mathbb{R}^k$ be  a nonempty compact set such that $X=cl{(int X)}.$ Then, $(X,d)$ is a compact metric space, where $d$ is the Euclidean metric on $X$. Let $N\geq2$ and $\{f_{ij}: 1\leq i,j\leq N\}$ be a system such that
	$$c_{ij}d(x,y)\leq d(f_{ij}(x),f_{ij}(y))\leq s_{ij}d(x,y) \te{ for all } x,y\in X,$$
	where $0<c_{ij}\leq s_{ij}<1$ for $1\leq i,j\leq N.$ Let $P=[p_{ij}]_{1\leq i,j\leq N}$ be an $N\times N$ irreducible row stochastic matrix, that is, for all $i$, $\sum_{j=1}^{N}p_{ij}=1$ and $p_{ij}\geq 0$ for all $1\leq i,j\leq N.$  Then,  the collection $\mathcal{I}=\{X; f_{ij}, p_{ij}: 1\leq i,j\leq N\}$ is called a bi-Lipschitz recurrent iterated function system (RIFS). By using the Banach's fixed point theorem, there exist unique nonempty compact subsets $E_1,E_2,\ldots,E_N$ of $X$ such that
	\begin{equation} \label{invariant}
		E_i=\bigcup_{j,p_{ji}>0}f_{ij}(E_j).
	\end{equation}
	The set $E=\bigcup_{i=1}^{N}E_i$ is called the limit set of the RIFS $\mathcal{I}.$
	\par For $n\geq 2$, let $\Omega_n=\{(w_1,w_2,\ldots,w_n): p_{w_{i+1}w_{i}}>0 \te{ for all } 1\leq i\leq n-1~~ \text{and}~~1\leq w_i\leq N\}$. Set $\Omega^*=\bigcup_{n\geq 2}\Omega_n$. For $w=(w_1,w_2,\ldots,w_n)\in \Omega^*$, we define $E_{w}=f_{w_1w_2}\circ f_{w_2w_3}\circ\dots\circ f_{w_{n-1}w_n}(E_{w_n})$. By using \eqref{invariant}, the limit set $E$ satisfies the following invariance equality
	\begin{equation}
		E=\bigcup_{w\in \Omega_n}E_w \te{ for all } n\geq2.
	\end{equation}
	Define the code space $\Omega$ associated with the RIFS $\mathcal{I}$ as follows
	$$\Omega=\{(w_1,w_2,\ldots): p_{w_{i+1}w_{i}}>0 \te{ for all } i\in \mathbb{N}~~ \text{and}~~1\leq w_i\leq N\}.$$
	Let $\Pi:\Omega\times \Omega \to \mathbb{R}$ be defined as follows
	$$\Pi(w,\tau)=\begin{cases}
		0 & \text{if}~~w=\tau,\\
		2^{-n} & \text{if}~~w\ne \tau,
	\end{cases}$$
	where $n=\min\{m: w_m \ne \tau_m\}.$
	Then,  $(\Omega,\Pi)$ is a compact metric space.
	\par
	For $w=(w_1,w_2,\ldots)\in \Omega$, we denote $w|_n=(w_1,w_2,\ldots, w_n)$ and $w|_0=\emptyset$. Similarly, for $w=(w_1,w_2,\ldots,w_m)\in \Omega^*$, if $n\leq m$, then we write $w|_n=(w_1,w_2,\ldots, w_n).$ For $w=(w_1,w_2,\ldots, w_n)$ and $\tau=(\tau_1,\tau_2,\ldots, \tau_m)$ in $\Omega^*$, if $p_{\tau_{1}w_n}>0$, we denote \[w\tau=(w_1,w_2,\ldots, w_n,\tau_1,\tau_2,\ldots, \tau_m)\in \Omega^*.\]
	We say that $w$ is an extension of $\tau\in \Omega^*$ if $w|_{|\tau|}=\tau$, where $|\tau|$ represents the length of $\tau.$ For any $w=(w_1,w_2,\ldots,w_n)\in \Omega^*$, we define $w^-=(w_1,w_2,\ldots,w_{n-1}).$
	\par
	We define the coding map  $\pi: \Omega\to E$  as follows
	$$\pi(w)=\bigcap_{n\geq2} f_{w_1w_2}\circ f_{w_2w_3}\circ\dots\circ f_{w_{n-1}w_n}(E_{w_n}),$$
	where $w=(w_1,w_2,\ldots)\in \Omega.$ By the definition of $\pi$, it is clear that $\pi$ is an onto and continuous mapping.
	\par If $w=(w_1,w_2,\ldots,w_n)\in \Omega^*$, then the set $\{\tau\in \Omega: \tau_i=w_i ~~\text{for}~~i\in \{1,2,\ldots,n\}\}$ is called the cylinder of length $n$ in $\Omega$  generated by $w\in \Omega^*$, and is denoted by $[w]=[w_1,w_2,\ldots,w_n].$ A cylinder of length $0$ is called the empty cylinder. The set of all elements in $\Omega$ starting with the symbol $i$ is denoted by $C(i).$
	\par Let $\mathcal{B}$ be the Borel sigma-algebra generated by the cylinders in $\Omega.$ Since $P$ is an irreducible row stochastic matrix, there exists a unique probability vector $p=(p_1,p_2,\ldots,p_N)$ such that $pP=p$, that is $\sum_{i=1}^{N}p_ip_{ij}=p_j $ for all $j\in \{1,2,\ldots,N\}.$ We define
	$$\nu([w_1,w_2,\ldots,w_n])=p_{w_2w_1}p_{w_3w_2}\ldots p_{w_{n}w_{n-1}}p_{w_n},$$
	where $[w_1,w_2,\ldots,w_n]$ is a cylinder in $\Omega.$ Using  Kolmogorov's extension theorem, $\nu$ can be extended to a unique Borel probability measure on $\mathcal{B},$ which we denote by $\nu$ itself. The support of the measure $\nu$ is $\Omega.$ For $w=(w_1,w_2,\ldots, w_n)\in \Omega^*$, $n\geq2$, we denote
	$$f_w=f_{w_1w_2}\circ f_{w_2w_3}\circ\dots\circ f_{w_{n-1}w_n},\quad c_w=c_{w_1w_2} c_{w_2w_3}\dots c_{w_{n-1}w_n},$$
	$$s_w=s_{w_1w_2} s_{w_2w_3}\dots s_{w_{n-1}w_n},\quad p_w=p_{w_2w_1} p_{w_3w_2}\dots p_{w_{n}w_{n-1}}p_{w_n}.$$
	Also, let $$s_{\max}=\max\{s_{ij}: p_{ji}>0, 1\leq i,j\leq N\},\quad s_{\min}=\min\{s_{ij}: p_{ji}>0, 1\leq i,j\leq N\},$$
	$$c_{\max}=\max\{c_{ij}: p_{ji}>0, 1\leq i,j\leq N\},\quad c_{\min}=\min\{c_{ij}: p_{ji}>0, 1\leq i,j\leq N\}$$
	$$p_{\max}=\max\{p_1,p_2,\ldots,p_N\},\quad p_{\min}=\min\{p_1,p_2,\ldots,p_N\},$$
	$$P_{\max}=\max\{p_{ij}: 1\leq i,j\leq N \}, \quad P_{\min}=\min\{p_{ij}>0: 1\leq i,j\leq N \}.$$
	Let $\mu$ be the push forward measure of $\nu$ under the coding map $\pi$ on the limit set $E$ of the RIFS $\mathcal{I},$ that is, $\mu=\nu\circ \pi^{-1}.$ Clearly, the support of $\mu$ is $E.$ Our aim is to determine the quantization dimension of this measure $\mu.$ Define
	$$\pi_i:=\pi|_{C(i)},\quad \mu_i=\mu|_{E_i}.$$
	Clearly $\mu_i=\nu\circ\pi_i^{-1}.$ For each $1\leq i\leq N$, the probability measure $\mu_i$ satisfies the  invariance
	$\mu_i=\sum_{j=1}^{N}p_{ji}\mu_j\circ f_{ij}^{-1}.$ 
	\begin{lemma}\label{le2.5}
		For each $m\geq 2$, $\mu$ satisfies the following inequality
		\begin{align}\label{eq4}
			\mu\leq \frac{1}{p_{\min}}\sum_{w\in \Omega_m}p_w\mu\circ f_w^{-1}.
		\end{align}
	\end{lemma}
	\begin{proof} First, for simplicity, we will prove the inequality for $m=2.$
		Let $A$ be a Borel subset of $\mathbb{R}^k.$ Thus, we have 
		$$\mu(A)=\mu(\cup_{i=1}^{N}(A\cap E_i))\leq \sum_{i=1}^{N}\mu(A\cap E_i)=\sum_{i=1}^{N}\mu_i(A\cap E_i)=\sum_{i=1}^{N}\sum_{j=1}^{N}p_{ji}\mu_j\circ f_{ij}^{-1}(A\cap E_i).$$
		This implies that 
		$$\mu(A)\leq \sum_{i=1}^{N}\sum_{j=1}^{N}p_{ji}\mu\circ f_{ij}^{-1}(A\cap E_i)\leq  \frac{1}{p_{\min}}\sum_{w\in \Omega_2}p_w\mu\circ f_w^{-1}(A).$$
		Thus, for $m=2,$ the inequality \eqref{eq4} holds. Now, we will prove inequality   \eqref{eq4} for $m > 2$. We have 
		\begin{align*}
			\mu_i=\sum_{j=1}^{N}p_{ji}\mu_j\circ f_{ij}^{-1}=\sum_{j=1}^{N}p_{ji}\bigg(\sum_{k=1}^{N}p_{kj}\mu_k\circ f_{jk}^{-1}\bigg)\circ f_{ij}^{-1}=\sum_{\substack{w=(i,j,k)\in \Omega_3\\1\leq j,k\leq N}}p_{ji}p_{kj}\mu_{k}\circ (f_{ij}\circ f_{jk})^{-1}.
		\end{align*}
		By  repeatedly applying the value of $\mu_k$ in the above expression, we get 
		\begin{align*}
			\mu_i&=\sum_{\substack{w=(i,w_2,w_3,w_4,\ldots,w_{m-1},w_m)\in \Omega_m}}p_{w_2 i}p_{w_3 w_2}\cdots p_{w_m w_{m-1}}\mu_{m}\circ (f_{i w_2}\circ f_{w_2 w_3}\circ \cdots f_{w_{m-1} w_m})^{-1}\\&=\sum_{\substack{w=(i,w_2,w_3,w_4,\ldots,w_{m-1},w_m)\in \Omega_m}}\frac{p_{w}}{p_{w_m}}\mu_{m}\circ f_{w}^{-1}
		\end{align*}
		Thus, by the above, we get
		\begin{align*}
			\mu(A)\leq \sum_{i=1}^{N}\mu_i(A\cap E_i)=& \sum_{i=1}^{N} \sum_{\substack{w=(i,w_2,w_3,w_4,\ldots,w_{m-1},w_m)\in \Omega_m}}\frac{p_{w}}{p_{w_m}}\mu_{m}\circ f_{w}^{-1}(A\cap E_i)\\ \leq& \sum_{i=1}^{N} \sum_{\substack{w=(i,w_2,w_3,w_4,\ldots,w_{m-1},w_m)\in \Omega_m}}\frac{p_{w}}{p_{w_m}}\mu \circ f_{w}^{-1}(A)\\\leq& \frac{1}{p_{\min}}\sum_{w\in \Omega_m}p_w\mu\circ f_w^{-1}(A).
		\end{align*}
		This completes the proof.
	\end{proof}
	Now, we define a matrix corresponding to the upper  contraction ratios of the mappings $f_{ij}$ for $1\leq i,j\leq N.$ Let us define an $N\times N$ matrix $M_{r,t}=[(p_{ji}s_{ij}^r)^t]_{1\leq i,j\leq N}$, where $r\in (0,\infty)$ and $t\in [0,\infty).$  Let  $\Phi_{r}(t)$ be the spectral radius of the matrix $M_{r,t}.$ Then, it is well-known that $\Phi_{r}(t)=\lim_{n\to \infty}{\|M_{r,t}^{n}\|}_{1}^{\frac{1}{n}}$ and  $\Phi_{r}(t)$ is the largest eigenvalue of $M_{r,t}$. Let $m_{ij}^{(n)}(r,t)$ be the $(i,j)$th entry of the matrix $M_{r,t}^{n},$ then
	$$m_{ij}^{(n)}(r,t)=\sum_{j_1,j_2,\ldots,j_{n-1}=1}^{N}
	\Big(p_{j_{1}i}p_{j_{2}j_1}\dots p_{j j_{n-1}}\Big(s_{ij_1}s_{j_1j_2}\dots s_{j_{n-1},j}\Big)^r\Big)^t.$$
	Thus, we have $$\Phi_{r}(t)=\lim_{n\to \infty}\Big( \sum_{i,j=1}^{N}m_{ij}^{(n)}(r,t)\Big)^\frac{1}{n}.$$
	Similarly, corresponding to the lower contraction ratios, we define a matrix $L_{r,t}=[(p_{ji}c_{ij}^r)^t]_{1\leq i,j\leq N}$ and $\phi_r(t)$ denotes the spectral radius of this matrix.
	\begin{defi}
		The recurrent IFS $\mathcal{I}=\{X; f_{ij}, p_{ij}: 1\leq i,j\leq N\}$ satisfies the strong open set condition (SOSC) if there exist nonempty bounded open sets $U_1,U_2,\ldots,U_N$ such that
		$$f_{ij}(U_j)\cap f_{kl}(U_l)=\emptyset~~\text{for}~~(i,j)\ne (k,l),~~f_{ij}(U_j)\subseteq U_i~~\text{and}~~E_i\cap U_i\ne \emptyset,~~\text{where}~~1\leq i,j,k,l\leq N.$$
	\end{defi}
	\begin{defi}
		The recurrent IFS $\mathcal{I}=\{X; f_{ij}, p_{ij}: 1\leq i,j\leq N\}$ satisfies the strong separation condition (SSC) if
		$$f_{ij}(E_j)\cap f_{kl}(E_l)=\emptyset~~\forall~~ (i,j)\ne (k,l)~~\text{and}~~p_{ji}p_{lk}>0,~~\text{where}~~1\leq i,j,k,l\leq N.$$
	\end{defi}
	A finite set $\Gamma \subset \Omega^*$ is said to be a finite maximal antichain if each sequence in $\Omega$ is an extension  of some element in $\Gamma$,  but no element of $\Gamma$ is an extension of another element of $\Gamma.$
	\par
	Let $\tau,w\in \Omega^*$ such that $\tau=(\tau_{1},\tau_{2},\ldots,\tau_{m})\in \Omega_m$, $w=(w_{1},w_{2},\ldots,w_{n})\in \Omega_n$ and $p_{\tau_{1}w_n}>0$. Then, we have
	$$s_{w\tau}=s_ws_{\tau}s_{w_n\tau_1}, \quad c_{w\tau}=c_wc_{\tau}c_{w_n\tau_1}, \quad p_{w\tau}=p_{w}p_{\tau}\frac{p_{\tau_{1}w_n}}{p_{w_{n}}}.$$
	By the above, we obtain
	$$ s_ws_{\tau}s_{\min}\leq s_{w\tau}\leq s_ws_{\tau}s_{\max},\quad c_wc_{\tau}c_{\min}\leq c_{w\tau}\leq c_wc_{\tau}c_{\max}, \quad p_{w}p_{\tau}\frac{P_{\min}}{p_{\max}} \leq
	p_{w\tau}\leq  p_{w}p_{\tau}\frac{P_{\max}}{p_{\min}}.$$
	Therefore, we get
	\begin{align}\label{2.3}
		(p_{w}s_{w}^r)(p_{\tau}s_{\tau}^r) A_r\leq p_{w\tau}s_{w\tau}^r\leq (p_{w}s_{w}^r)(p_{\tau}s_{\tau}^r)\Tilde{A}_r,\end{align}
	\begin{align}\label{2.4}
		(p_{w}c_{w}^r)(p_{\tau}c_{\tau}^r) B_r\leq p_{w\tau}c_{w\tau}^r\leq (p_{w}c_{w}^r)(p_{\tau}c_{\tau}^r)\Tilde{B}_r,
	\end{align}
	where $A_r=\min\{s_{\min}^r\frac{P_{\min}}{p_{\max}},1\},\quad  \Tilde{A}_r=\max\{s_{\max}^r \frac{P_{\max}}{p_{\min}},1\}, \quad B_r=\min\{c_{\min}^r\frac{P_{\min}}{p_{\max}},1\}$  and $ \Tilde{B}_r=\max \{c_{\max}^r \frac{P_{\max}}{p_{\min}},1\}.$
	\begin{remark}
		In the above construction of the bi-Lipschitz RIFS, if we put $f_{ij}=f_i, c_{ij}=c_i, s_{ij}=s_i$ and $p_{ij}=p_j$ for all $1\leq i,j\leq N,$ then the bi-Lipschitz RIFS reduces to the bi-Lipschitz IFS $\{X;f_1,f_2,\ldots,f_N\}$ with the probability vector $(p_1,p_2,
		\ldots,p_N).$  One can easily see that the limit set $E$ satisfies $E=\cup_{i=1}^{N}f_i(E)$ and $\mu$ becomes the invariant measure corresponding to this IFS such that $\mu=\sum_{i=1}^{N}p_i\mu\circ f_{i}^{-1}.$
		In this case, the matrix $M_{r,t}=[(p_is_{i}^{r})^t]_{1\leq i,j\leq N}.$ It is not difficult to see that the matrix $M_{r,t}$ is a rank-1 matrix. So, $0$ is an eigenvalue of $M_{r,t}$ with geometric multiplicity $N-1.$ This implies that the spectral radius $\Phi_{r}(t)=\sum_{i=1}^{N}(p_is_{i}^r)^t.$ Similarly, $\phi_{r}(t)=\sum_{i=1}^{N}(p_ic_{i}^r)^t.$
	\end{remark}
	
	\section{Main result}\label{se3}
	In the following theorem, which is the main theorem of the paper, we  determine the upper and the lower bounds of the quantization dimensions of the Borel probability measures supported on the limit sets of the bi-Lipschitz recurrent  IFS under the strong open set condition in terms of the spectral radius.
	\begin{theorem}\label{mainthm}
		Let $\mathcal{I}=\{X; f_{ij}, p_{ij}: 1\leq i,j\leq N\}$ be a bi-Lipschitz recurrent  IFS satisfying the strong open set condition such that $$c_{ij}d(x,y)\leq d(f_{ij}(x),f_{ij}(y))\leq s_{ij}d(x,y) \quad \forall \quad x,y\in X,$$
		where $0<c_{ij}\leq s_{ij}<1$ for $1\leq i,j\leq N.$
		Let $\mu$ be the Borel probability measure supported on the limit set $E$ of the RIFS $\mathcal{I}$ defined as in the above section. Then,
		$$l_r\leq\underline{D}_r(\mu)\leq \overline{D}_r(\mu)\leq k_r,$$
		where $l_r$ and $k_r$ are uniquely given by $\phi_{r}\Big(\frac{l_r}{r+l_r}\Big)=1$ and $\Phi_{r}\Big(\frac{k_r}{r+k_r}\Big)=1$, respectively.
	\end{theorem}
	\begin{remark}\label{clear}
		In \cite{GL2}, Graf and Luschgy gave the formula \eqref{for} for the quantization dimension of the invariant measure of the self-similar IFS under the OSC. However, in the above theorem, if we consider $f_{ij}=f_i,$ $p_{ij}=p_j$ and $c_{ij}=s_{ij}=s_i$, then the bi-Lipschitz recurrent IFS reduces to the self-similar IFS with the OSC, and we get the same formula \eqref{for} for the quantization dimension. Thus, our result is a general version of the Graf-Luschgy's result.
	\end{remark}
	
	\begin{remark}
		In \cite{R2}, Roychowdhury determined the upper and the lower bound of the quantization dimension of the invariant probability measure of the bi-Lipschitz IFS under the SOSC and some restrictions on the bi-Lipschitz constants. Note that in the above theorem, we are not taking any restrictions on the bi-Lipschitz constants as in \cite{R2}. By Remark \ref{clear} and the above theorem, we obtain the same bound of the quantization dimension as in Roychowdhury's paper \cite{R2}. Thus, our result generalizes the result of  Roychowdhury \cite{R2} in a more general setting.
	\end{remark}
	
	\begin{remark}\label{re3.4}
		In \cite{MK1}, the author determined the upper and the lower bound of the quantization dimension for the measure supported on the limit set of the hyperbolic (bi-Lipschitz) RIFS.  For proving the upper and the lower bound, the author claims that the measure $\mu$ supported on the  limit set of the hyperbolic (bi-Lipschitz) RIFS satisfies the following invariance relation
		\begin{equation}\label{eq10}
			\mu=\sum_{w\in \Omega_n}p_w\mu\circ f_w^{-1}	
		\end{equation}
		for all $n\geq 2$. However, the above invariance relation is not correct as shown below. Let us take a simple case that the RIFS satisfies the SSC. Let $w=(w_1,w_2,\ldots,w_{n+1})\in \Omega_{n+1}.$ Set $\tau:=w_{|_n}=(w_1,w_2,\ldots,w_n)$ and $\sigma_0:=(w_n,w_{n+1})$. Then, we have
		$$\mu(E_{w})=\nu\circ \pi^{-1}(E_{w})=\nu([w_1,w_2,\ldots,w_{n+1}])=p_{w}.$$ On the other hand, 
		$$\sum_{\sigma\in \Omega_n}p_{\sigma}\mu\circ f_{\sigma}^{-1}(E_w)= p_{\tau} \mu\circ f_{\tau}^{-1}(E_w)=p_{\tau} \mu(E_{\sigma_0})=p_{\tau} p_{\sigma_0}=p_{w_n}p_{w}.$$
		This implies that the invariance relation \eqref{eq10} does not hold for this simple case. Thus, the relation \eqref{eq10} is not valid for RIFS. In \cite{R9}, the author used the same invariance relation \eqref{eq10} for the measure $\mu$  supported on the  limit set of the self-similar RIFS in determining  the quantization dimension of the measure $\mu,$ that is also not valid. In this paper, we modify the invariance relation for the measure $\mu$ and provide a different technique to determine the quantization dimension of the measure $\mu$ supported on the limit set of the bi-Lipschitz RIFS. Since the self-similar RIFS is a particular case of the bi-Lipschitz RIFS, this paper can be treated as the corrected version of the papers  \cite{R9} and \cite{MK1}.
	\end{remark}
	In order to prove Theorem~\ref{mainthm}, we need some lemmas and propositions.
	\par
	
	\begin{lemma}\label{unique}
		Let $r\in (0,\infty)$ be fixed. Then,  there exist unique $k_r,l_r\in (0,\infty)$ such that
		$$\Phi_r\Big(\frac{k_r}{r+k_r}\Big)=1\quad\text{and}\quad  \phi_r\Big(\frac{l_r}{r+l_r}\Big)=1.$$
	\end{lemma}
	\begin{proof}
		The mapping $\Phi_r: [0,\infty)\to [0,\infty)$ is defined as
		$$\Phi_r(t)=\lim_{n\to \infty}\Big( \sum_{i,j=1}^{N}m_{ij}^{(n)}(r,t)\Big)^\frac{1}{n},$$
		where $m_{ij}^{(n)}(r,t)$ defined as in the above section. Clearly, $\Phi_r$ is a continuous and strictly decreasing mapping,  and
		$$\Phi_r(0)=\lim_{n\to \infty}\Big( \sum_{i,j=1}^{N}m_{ij}^{(n)}(r,0)\Big)^\frac{1}{n}=\lim_{n\to \infty}\Big(\sum_{i,j=1}^{N} N^{n-1} \Big)^\frac{1}{n}=N\geq 2.$$
		Moreover, \begin{align*}
			\Phi_r(1)&=\lim_{n\to \infty}\Big( \sum_{i,j=1}^{N}m_{ij}^{(n)}(r,1)\Big)^\frac{1}{n}\\&=\lim_{n\to \infty}\Big(\sum_{i,j=1}^{N}\sum_{j_1,j_2,\dots,j_{n-1}=1}^{N}
			(p_{j_{1}i}p_{j_{2}j_1}\dots p_{j j_{n-1}}(s_{ij_1}s_{j_1j_2}\dots s_{j_{n-1},j})^r)\Big)^\frac{1}{n}\\& \leq s_{\max}^r<1.
		\end{align*}
		Since $\Phi_r$ is continuous, strictly decreasing, $\Phi_r(0)\geq2$ and $\Phi_r(1)<1$, there exists a unique $t_r\in (0,1)$ such that $\Phi_r(t_r)=1.$ Since $t_r\in (0,1)$, there exists a unique $k_r\in (0,\infty)$ such that $t_r=\frac{k_r}{r+k_r}$.
		Similarly, there exits a unique $l_r\in (0,\infty)$ such that $\phi_r\Big(\frac{l_r}{r+l_r}\Big)=1.$ This completes the proof.
	\end{proof}
	\begin{lemma}
		Let $r\in (0,\infty)$ be fixed and let $k_r,l_r$ be defined as in the above lemma. Then,
		$$\lim_{n\to \infty}\frac{1}{n}\log\sum_{w\in \Omega_n}(p_w {c_w}^r)^\frac{l_r}{r+l_r}=0\quad \text{and}\quad  \lim_{n\to \infty}\frac{1}{n}\log\sum_{w\in \Omega_n}(p_w {s_w}^r)^\frac{k_r}{r+k_r}=0.$$
	\end{lemma}
	\begin{proof}
		Consider the mapping  $\theta_r: [0,\infty)\to \mathbb{R}$ defined by $\theta_r(t)=\lim_{n\to \infty}\frac{1}{n}\log\sum_{w\in \Omega_n}(p_w {s_w}^r)^t.$ By \eqref{2.3}, one can easily see that $\theta_r(t)$ exists. Thus,
		\begin{align*}
			\theta_r(t)&=\lim_{n\to \infty}\frac{1}{n}\log\Big(\sum_{i,j_1,j_2,\dots,j_{n-1},j=1}^{N}(p_{j_{1}i}p_{j_{2}j_1}\dots p_{j j_{n-1}}p_{j} ({s_{ij_1}s_{j_1j_2}\dots s_{j_{n-1}j}})^r)^t\Big)\\&=  \lim_{n\to \infty}\frac{1}{n}\log\Big(\sum_{i,j=1}^{N}p_j^{t}m_{ij}^{(n)}(r,t)\Big).
		\end{align*}
		Therefore, we get
		$$\lim_{n\to \infty}\frac{1}{n}\log\Big(\sum_{i,j=1}^{N}p_{\min}^{t}m_{ij}^{(n)}(r,t)\Big)\leq \theta_r(t)\leq \lim_{n\to \infty}\frac{1}{n}\log\Big(\sum_{i,j=1}^{N}p_{\max}^{t}m_{ij}^{(n)}(r,t)\Big),$$
		$$\lim_{n\to \infty}\frac{1}{n}\log\Big(\sum_{i,j=1}^{N}m_{ij}^{(n)}(r,t)\Big)\leq \theta_r(t)\leq \lim_{n\to \infty}\frac{1}{n}\log\Big(\sum_{i,j=1}^{N}m_{ij}^{(n)}(r,t)\Big).$$
		Thus, we obtain $\log(\Phi_r(t))=\theta_r(t).$
		By using the previous lemma, we get $\theta_r(\frac{k_r}{r+k_r})=0.$ By using the similar arguments as above, we obtain $\lim\limits_{n\to \infty}\frac{1}{n}\log\sum\limits_{w\in \Omega_n}(p_w {c_w}^r)^\frac{l_r}{r+l_r}=0.$
		Thus, the proof is complete.
	\end{proof}
	\begin{lemma}\label{bb}
		Let $r\in (0,\infty)$ be fixed and let $k_r,l_r$ be defined as in Lemma \ref{unique}. Then, for each $n\geq 2$
		$$F_{r}^{\frac{-k_r}{r+k_r}}\leq \sum_{w\in \Omega_n}(p_ws_{w}^r)^{\frac{k_r}{r+k_r}}\leq F_{r}^{\frac{k_r}{r+k_r}} \text{ and }  G_{r}^{\frac{-l_r}{r+l_r}}\leq \sum_{w\in \Omega_n}(p_wc_{w}^r)^{\frac{l_r}{r+l_r}}\leq G_{r}^{\frac{l_r}{r+l_r}},$$
		where $F_r=\max\{A_{r}^{-1},\Tilde{A}_r\}$ and $G_r=\max\{B_{r}^{-1},\Tilde{B}_r\}$.
	\end{lemma}
	\begin{proof}
		We consider the same mapping $\theta_r$ as in the previous lemma. Then, we have $\theta_r(t)=\lim_{n\to \infty}\frac{1}{n}\log\sum_{w\in \Omega_n}(p_w {s_w}^r)^t.$
		We can also write $\theta_r(t)$ as follows
		$$\theta_r(t)=\lim_{m\to \infty}\frac{1}{nm}\log\sum_{w\in \Omega_{nm}}(p_w {s_w}^r)^t~~\forall~~ n\in \mathbb{N}.$$
		Thus, by the above and inequality \eqref{2.3}, we get
		$$\lim_{m\to \infty}\frac{1}{nm}\log(\sum_{w\in \Omega_{n}}(p_w {s_w}^r {A}_r)^t)^m\leq \theta_r(t)\leq \lim_{m\to \infty}\frac{1}{nm}\log(\sum_{w\in \Omega_{n}}(p_w {s_w}^r \Tilde{A}_r)^t)^m.$$
		This implies that
		$$\frac{1}{n}\log\sum_{w\in \Omega_{n}}(p_w {s_w}^r {A}_r)^t\leq \theta_r(t)\leq \frac{1}{n}\log\sum_{w\in \Omega_{n}}(p_w {s_w}^r \Tilde{A}_r)^t.$$
		Therefore, we obtain
		$$e^{n \theta_r(t)}\Tilde{A}_r^{-t}\leq \sum_{w\in \Omega_{n}}(p_w {s_w}^r )^t\leq e^{n \theta_r(t)}A_r^{-t}.$$
		Let $F_r=\max\{A_{r}^{-1},\Tilde{A}_r\}$. By the previous lemma, $\theta_r(\frac{k_r}{r+k_r})=0.$ Thus, we get
		$$F_{r}^{\frac{-k_r}{r+k_r}}\leq \sum_{w\in \Omega_n}(p_ws_{w}^r)^{\frac{k_r}{r+k_r}}\leq F_{r}^{\frac{k_r}{r+k_r}}.$$
		Similarly, we can get the other inequality by taking $c_w$ instead of $s_w$ in the above. This completes the proof.
	\end{proof}
	\begin{lemma}\label{lemboundfinite}
		Let $r\in (0,\infty)$ be fixed and let $\Gamma$ be a finite maximal antichain. Then, there exist constants $\eta_r\geq 1$ and  $\Tilde{\eta}_r\geq 1$ such that
		$$\sum_{w\in \Gamma}(p_ws_{w}^r)^{\frac{k_r}{r+k_r}}\leq \eta_r\quad\text{and} \quad\sum_{w\in \Gamma}(p_wc_{w}^r)^{\frac{k_r}{r+k_r}}\leq \Tilde{\eta}_r.$$
	\end{lemma}
	\begin{proof}
		Since $\Gamma$ is a finite maximal antichain,  we have $\Gamma=\bigcup_{i=1}^{K}\Gamma_{n_i}$ for some $K>0$, where $\Gamma_{n_i}=\{w\in \Gamma: |w|=n_i\}$ for all $i\in \{1,2,\dots,K\}$. Let $n> 1+\max\{n_1,n_2,\dots,n_K\}.$ Then,  by the previous lemma, we have
		\begin{align*}
			F_{r}^{\frac{k_r}{r+k_r}}&\geq  \sum_{\tau\in \Omega_n}(p_\tau s_{\tau}^r)^{\frac{k_r}{r+k_r}}=
			\sum_{i=1}^{K}\sum_{w\in\Gamma_{n_i}}\sum_{\tau\in \Omega_n, w\mathcal{<}\tau} (p_\tau s_{\tau}^r)^{\frac{k_r}{r+k_r}}=
			\sum_{i=1}^{K}\sum_{w\in\Gamma_{n_i}}\sum_{\tau\in \Omega_{n-n_i}} (p_{w\tau} s_{w\tau}^r)^{\frac{k_r}{r+k_r}}\\&=\sum_{i=1}^{K}\sum_{w\in\Gamma_{n_i}}(p_ws_{w}^r)^{\frac{k_r}{r+k_r}}\sum_{\tau\in \Omega_{n-n_i}} \Big(p_{\tau_{1}w_{n_i}}\frac{p_{\tau}}{p_{w_{n_i}}}s_{w_{n_i}\tau_1}^rs_{\tau}^r\Big)^{\frac{k_r}{r+k_r}}\\&\geq \sum_{i=1}^{K}\sum_{w\in\Gamma_{n_i}}(p_ws_{w}^r)^{\frac{k_r}{r+k_r}}(P_{\min}p_{\max}^{-1}s_{\min}^{r})^{\frac{k_r}{r+k_r}}\sum_{\tau\in \Omega_{n-n_i}} (p_{\tau}s_{\tau}^r)^{\frac{k_r}{r+k_r}}.
		\end{align*}
		Again, using the previous lemma, we get
		$$ F_{r}^{\frac{k_r}{r+k_r}}\geq \sum_{i=1}^{K}\sum_{w\in\Gamma_{n_i}}(p_ws_{w}^r)^{\frac{k_r}{r+k_r}}(P_{\min}p_{\max}^{-1}s_{\min}^{r})^{\frac{k_r}{r+k_r}}F_{r}^{\frac{-k_r}{r+k_r}}.$$
		This implies that
		$$\sum_{w\in \Gamma}(p_ws_{w}^r)^{\frac{k_r}{r+k_r}}\leq \eta_r,$$
		where $\eta_r=\max\Big\{1, F_{r}^{\frac{2k_r}{r+k_r}}(P_{\min}p_{\max}^{-1}s_{\min}^{r})^{\frac{-k_r}{r+k_r}}\Big\}.$ Similarly, one can get the other inequality.
		Thus, the proof is complete.
	\end{proof}
	\begin{lemma}
		Let $r\in (0,\infty)$ and let $\mu$ be the Borel probability measure supported on the limit set $E$ of the RIFS $\mathcal{I}$. Let $\Gamma\subset \Omega^*$ be a finite maximal antichain. Then,  for each $n\in\mathbb{N}$ with $n\geq |\Gamma|$, we have  $$V_{n,r}(\mu)\leq \inf \Big\{\frac{1}{p_{\min}}\sum_{w\in \Gamma}p_w s_{w}^rV_{n_w,r}(\mu):~~n_w\in \mathbb{N}, \sum\limits_{w\in \Gamma}n_w\leq n\Big\}.$$ \end{lemma}
	\begin{proof}
		Since $\Gamma\in \Omega^*$ is a finite maximal antichain, we have $\mu\leq \frac{1}{p_{\min}}\sum_{w\in \Gamma}p_w\mu\circ f_w^{-1}$. Let $A_{n_w}$ be an $n_w$-optimal set of $V_{n_w,r}(\mu)$
		for each $w\in \Gamma.$
		Since $f_w$ is a bi-Lipschitz map, we have $|f_w(A_{n_w})|= n_w $ and $|\cup_{w\in \Gamma}f_w(A_{n_w})|\leq n.$ Then, we have
		\begin{align*}
			V_{n,r}(\mu)&\leq \int d \Big(x,\bigcup_{w\in\Gamma}f_w(A_{n_w})\Big)^r d{\mu}(x)\\&\leq \frac{1}{p_{\min}}
			\sum_{w\in \Gamma}p_w
			\int d \Big(x,\bigcup_{w\in\Gamma}f_w(A_{n_w})\Big)^r d{(\mu\circ f_{w}^{-1})}(x)\\&= \frac{1}{p_{\min}}\sum_{w\in \Gamma}p_w
			\int d \Big(f_{w}(x),\bigcup_{w\in\Gamma}f_w(A_{n_w})\Big)^r d{\mu}(x)\\&\leq \frac{1}{p_{\min}}
			\sum_{w\in \Gamma}p_w
			\int d \Big(f_{w}(x),f_w(A_{n_w})\Big)^r d{\mu}(x)\\&\leq \frac{1}{p_{\min}}
			\sum_{w\in \Gamma}p_w s_{w}^r
			\int d (x,A_{n_w})^r d{\mu}(x)\\&= \frac{1}{p_{\min}}\sum_{w\in \Gamma}p_w s_{w}^r V_{n_w,r}(\mu).
		\end{align*}
		This completes the proof.
	\end{proof}
	In the following proposition, we prove that the upper quantization coefficient for the Borel probability measure supported on the limit set of the RIFS $\mathcal{I}$ is finite without assuming any separation condition.
	\begin{prop}\label{prop3.7}
		Let $\mathcal{I}=\{X; f_{ij}, p_{ij}: 1\leq i,j\leq N\}$ be a bi-Lipschitz recurrent  IFS. Let $\mu$ be the Borel probability measure supported on the limit set $E$ of the RIFS $\mathcal{I}$. Let $r\in (0,\infty)$ and let $k_r\in (0,\infty)$ be a unique number such that $\Phi_r(\frac{k_r}{r+k_r})=1.$ Then,
		$$\limsup_{n\to \infty}n e_{n,r}^{k_r}(\mu)<\infty.$$
	\end{prop}
	\begin{proof}
		Let $\xi_{\min}=\min\{(p_{j}p_{ji}s_{ij}^r)^{\frac{k_r}{r+k_r}}: 1\leq i,j\leq N\}$. Clearly $\xi_{\min}\in (0,1).$ Let $n_0\in \mathbb{N}$ be a fixed number and $\eta_r\geq 1$ be as in Lemma \ref{lemboundfinite}. Take any $n\in \mathbb{N}$ such that $\frac{n_{0}\eta_r}{n}\leq \xi_{\min}^2.$ This holds for all but finitely many  values of $n.$ Let $\epsilon=\frac{n_{0}\eta_r}{n}\xi_{\min}^{-1}.$ Clearly $\epsilon\in (0,1).$ Now, we define a subset $\Gamma_{\epsilon}$ of $\Omega^*$ as follows
		$$\Gamma_{\epsilon}=\bigg\{w=(w_1,w_2,\ldots,w_n)\in \Omega^*: \bigg(\frac{p_{w^-}}{p_{w_{n-1}}}s_{w^-}^r\bigg)^{\frac{k_r}{r+k_r}}\geq \epsilon > \bigg(\frac{p_{w}}{p_{w_n}}s_{w}^r\bigg)^{\frac{k_r}{r+k_r}}\bigg\}.$$
		Then,  $\Gamma_{\epsilon}$ is a finite maximal antichain. Using Lemma \ref{lemboundfinite}, we have
		\begin{align*}
			\eta_r\geq \sum_{w\in \Gamma_{\epsilon}}(p_ws_{w}^r)^{\frac{k_r}{r+k_r}}&=  \sum_{w=(w_1,w_2,\ldots,w_n)\in \Gamma_{\epsilon}}\bigg(\frac{p_{w^-}}{p_{w_{n-1}}}s_{w^-}^r p_{w_n}p_{w_{n}w_{n-1}}s_{w_{n-1}w_{n}}^{r}\bigg)^{\frac{k_r}{r+k_r}}\\&\geq  \sum_{w=(w_1,w_2,\ldots,w_n)\in \Gamma_{\epsilon}}\bigg(\frac{p_{w^-}}{p_{w_{n-1}}}s_{w^-}^r\bigg)^{\frac{k_r}{r+k_r}}\xi_{\min}\\&\geq |\Gamma_{\epsilon}|\xi_{\min}\epsilon,
		\end{align*}
		which implies that $|\Gamma_{\epsilon}|\leq \frac{n}{n_0}$. Thus, we have  $\sum_{w\in \Gamma_{\epsilon}}n_0\leq n$. Therefore, by the above lemma, we get
		\begin{align*}
			V_{n,r}(\mu)&\leq \frac{1}{p_{\min}} \sum_{w\in \Gamma_{\epsilon}}p_w s_{w}^r V_{n_0,r}(\mu)\\&= \frac{1}{p_{\min}}\sum_{w\in \Gamma_{\epsilon}}(p_w s_{w}^r)^{\frac{k_r}{r+k_r}}(p_w s_{w}^r)^{\frac{r}{r+k_r}} V_{n_0,r}(\mu)\\&\leq \frac{1}{p_{\min}} \sum_{w\in \Gamma_{\epsilon}}(p_w s_{w}^r)^{\frac{k_r}{r+k_r}}{\epsilon}^{\frac{r}{k_r}} V_{n_0,r}(\mu).
		\end{align*}
		Using Lemma  \ref{lemboundfinite} and $\epsilon=\frac{n_{0}\eta_r}{n}\xi_{\min}^{-1}$, we obtain
		$$V_{n,r}(\mu)\leq \frac{1}{p_{\min}}\eta_r \Big({\frac{n_{0}\eta_r}{n}\xi_{\min}^{-1}}\Big)^{\frac{r}{k_r}} V_{n_0,r}(\mu),$$
		which implies that
		$$n e_{n,r}^{k_r}(\mu)\leq \bigg(\frac{1}{p_{\min}}\bigg)^{\frac{k_r}{r}} \eta_{r}^{(\frac{k_r}{r}+1)}n_0\xi_{\min}^{-1}e_{n_0,r}^{k_r}(\mu).$$
		The above inequality holds for all but finitely many values of $n.$ Thus, we have
		$$\limsup_{n\to \infty}n e_{n,r}^{k_r}(\mu)<\infty,$$
		which completes the proof of the proposition.
	\end{proof}
	In the end, we will use the above proposition to determine the upper bound of the quantization dimension. Now, we give some lemmas and propositions which will be used to determine the lower bound of the quantization dimension.
	\begin{lemma}
		Let $\mathcal{I}=\{X; f_{ij}, p_{ij}: 1\leq i,j\leq N\}$ be a bi-Lipschitz recurrent  IFS. Let $\mu$ be the Borel probability measure supported on the limit set $E.$ Then, for each $\epsilon>0$
		$$\inf\{\mu(B(x,\epsilon)): x\in E\}>0.$$
	\end{lemma}
	\begin{proof}
		Since $t\to \sum\limits_{i,j=1,p_{ji}>0}^{N}s_{ij}^t$  is a strictly decreasing continuous function for $t\in [0,\infty)$, there exists a unique $t_o\in (0,\infty)$ such that $\sum\limits_{i,j=1,p_{ji}>0}^{N}s_{ij}^{t_0}=1.$ Let $\xi_{0}=(\frac{\epsilon}{\text{diam}(E)})^{t_0}.$ We define a finite maximal antichain $\Gamma_{\epsilon}$ as follows
		$$\Gamma_{\epsilon}=\{w\in \Omega^*: s_{w^-}^{t_0}\geq \xi_{0}>s_{w}^{t_0}\}.$$
		Since $E=\cup_{i=1}^{N}E_i$, we have $E=\cup_{w\in \Gamma_{\epsilon}}E_w.$ Let $x\in E$. Then, there exists a $w_{0}=(w_1,w_2,\ldots,w_n)\in \Gamma_{\epsilon}$ such that $x\in E_{w_0}=f_{w_0}(E_{w_n}).$ Let $x_1,x_2\in E_{w_n}.$ Then, we have
		\begin{align*}
			d(f_{w_0}(x_1),f_{w_0}(x_2))\leq s_{w_0}\text{diam}(E)< \epsilon.
		\end{align*}
		This implies that, $\text{diam}(E_{w_0})< \epsilon $. Since $x\in E_{w_0},$ we have $E_{w_0}\subseteq B(x,\epsilon).$ Since $\mu=\nu\circ \pi^{-1},$ we have
		\begin{align*}
			\mu{(E_{w_0})}= \nu\circ \pi^{-1}(E_{w_0})\geq \nu([w_1,w_2,\ldots,w_n])= p_{w_0}\geq \min_{w\in \Gamma_{\epsilon}} \{p_w\} .
		\end{align*}
		This implies that $\mu(B(x,\epsilon))\geq \min_{w\in \Gamma_{\epsilon}} \{p_w\}$. Since this holds for all $x\in E$, we have $\inf\{\mu(B(x,\epsilon)): x\in E\}>0.$ This completes the proof.
	\end{proof}
	\begin{prop}\label{SSC}
		Let $\mathcal{I}=\{X; f_{ij}, p_{ij}: 1\leq i,j\leq N\}$ be a bi-Lipschitz recurrent IFS satisfying the strong open set condition. Then, for each  $i\in \{1,2,\ldots,N\}$, there exists a sequence ${\sigma}^{i}=(i,\ldots,i)\in \Omega^*$ such that for each $m\in \mathbb{N}$, the RIFS
		$\mathcal{I}^m:=\{X; g_{w}, p_{ij}: 1\leq i,j\leq N, w\in
		\Omega_{m} \}$ satisfying the strong separation condition, where  $g_{w}:=f_{w} \circ f_{{\sigma}^{i}}$ for $w=(w_1, w_2,\ldots,w_{m_1},i)\in \Omega_m$.  Let $\mu_m$ be the Borel probability measure supported on the limit set $F^{m}$ of the RIFS $\mathcal{I}^m.$ Then, there exists $n_0\in \mathbb{N}$ such that for all $n\geq n_0$ there exist $n_w\in \mathbb{N}$ such that $\sum_{w\in \Omega_m}n_w\leq n$ and
		$$V_{n,r}(\mu_m)\geq \sum_{w\in \Omega_m}p_w C_{w}^{r}V_{n_w,r}(\mu_m),$$
		where $C_w:=c_{w}c_{{\sigma}^{i}}\in (0,1)$ for $w=(w_1, w_2,\ldots,w_{m_1},i)\in \Omega_m.$
	\end{prop}
	\begin{proof}
		Since the RIFS $\mathcal{I}$ satisfies the strong open set condition, there exist nonempty bounded open sets $U_1,U_2,\ldots, U_N$ such that
		$$f_{ij}(U_j)\cap f_{kl}(U_l)=\emptyset~~\text{for}~~(i,j)\ne (k,l),~~f_{ij}(U_j)\subseteq U_i~~\text{and}~~E_i\cap U_i\ne \emptyset~~\text{where}~~1\leq i,j,k,l\leq N.$$
		This also implies that for any
		$w=(w_1,w_2,\dots,w_m)$ and $\tau=(\tau_1,\tau_2,\ldots,\tau_m)$ in $\Omega_m$ such that $w\ne \tau$, we
		have $f_w(U_{w_m})\cap f_\tau(U_{\tau_m})=\emptyset.$
		\par
		For each $i\in \{1,2,\ldots,N\}$, given that $E_i\cap U_i\ne \emptyset,$ there exists a sequence $\sigma=(i,\sigma_2,\ldots,\sigma_n)\in \Omega^*$ such that $E_{\sigma}=f_{\sigma}(E_{\sigma_n})\subseteq U_i.$ Since $P$ is an irreducible matrix, for $i,\sigma_n$, there exist $j_1,j_2,\ldots,j_k\in \{1,2,\ldots,N\}$ such that $p_{ij_1}p_{j_1j_2}\dots p_{j_k\sigma_n}>0$. From this, we construct a finite sequence which starts from $\sigma_n$ and ends at $i,$ say $\sigma*=(\sigma_n,j_{k},j_{k-1},\ldots,j_{2},j_{1},i)\in \Omega^*.$ Now, we define a finite sequence $\sigma^i=(i,\sigma_2,\ldots,\sigma_n,j_{k},j_{k-1},\ldots,j_{2},j_{1},i).$ Clearly, $\sigma^i\in \Omega^*$ and $E_{\sigma^i}=f_{\sigma^i}(E_i)=f_{i\sigma_2}\circ f_{\sigma_2\sigma_3}\circ\dots \circ f_{\sigma_{n-1}\sigma_n}\circ f_{\sigma_n{j_k}}\circ\dots \circ f_{j_{1}i}(E_i)\subseteq U_i.$
		\par
		For each $m\in \mathbb{N},$ we define an RIFS  $\mathcal{I}^m=\{X; g_{w}, p_{ij}: 1\leq i,j\leq N, w\in
		\Omega_{m} \}$, where $g_w$ is defined as $g_w=f_w\circ f_{\sigma^i}$ for $w=(w_1,w_2,\dots,w_{m-1},i)\in \Omega_m.$ For $x,y\in X$, we have
		$$C_w d(x,y)\leq d(g_w(x),g_w(y))\leq S_w d(x,y),$$
		where $C_w=c_wc_{\sigma^i}, S_w=s_ws_{\sigma^i}$ and $w=(w_1,w_2,\dots,w_{m-1},i)\in \Omega_m.$ Now, we claim that the RIFS $\mathcal{I}^m$ satisfies the strong separation condition.
		Let $F^m$ be the limit set of the RIFS $\mathcal{I}^m$. Then,
		$$F^m=\bigcup_{i=1}^{N}F_{i}^{m}\quad \text{and}\quad F_{i}^{m}=\bigcup_{w\in \Omega_m, w=(i,w_2,\dots,w_m)}g_{w}(F_{w_m}^{m}).$$
		By observing the code spaces of $\mathcal{I}^m$ and $\mathcal{I}$, it is clear that the code space of $\mathcal{I}^m$ is contained in $\Omega.$ Then, we have $F^m\subseteq E$ and $F_{i}^{m}\subseteq E_i$ for each $i\in \{1,2,\dots,N\}.$ Let $w=(w_1,w_2,\dots,w_m)\ne \tau=(\tau_1,\tau_2,\dots,\tau_m)\in \Omega_m.$ Then, we have
		\begin{align*}
			g_w(F_{w_m}^m)=f_w\circ f_{\sigma^{w_m}}(F_{w_m}^m)\subseteq f_w\circ f_{\sigma^{w_m}}(E_{w_m})\subseteq f_w (U_{w_m})
		\end{align*}
		and
		\begin{align*}
			g_\tau(F_{\tau_m}^m)=f_\tau\circ f_{\sigma^{\tau_m}}(F_{\tau_m}^m)\subseteq f_\tau\circ f_{\sigma^{\tau_m}}(E_{\tau_m})\subseteq f_\tau (U_{\tau_m}).
		\end{align*}
		Since $w\ne \tau\in \Omega_m$, we obtain $ g_w(F_{w_m}^m)\cap g_\tau(F_{\tau_m}^m)=\emptyset.$ This implies that the RIFS $\mathcal{I}^m$ satisfies the strong separation condition. Thus, the claim is true.  We have
		\begin{align*}
			F^m&=\bigcup_{i=1}^{N}\bigcup_{w\in \Omega_m, w=(i,w_2,\dots,w_m)}g_{w}(F_{w_m}^{m})\\&=\bigcup_{i=1}^{N}\bigcup_{w\in \Omega_{m},\sigma\in \Omega_{m}, w=(i,w_2,\dots,w_m),\sigma=(w_m,\sigma_2,\sigma_3,\ldots,\sigma_m)}g_{w}\circ g_{\sigma}(F_{w_\sigma}^{m}).
		\end{align*} 
		Since the RIFS $\mathcal{I}^m$ satisfies the SSC, we can define a Borel probability measure $\mu_m$ (like a mass distribution) supported on the limit set $F^m$ as follows:\\
		At the $1$st  level of $F^{m}$,  $\{g_{w}(F_{w_m}^{m})\}_{w=(w_1,w_2,\dots,w_m)\in \Omega_{m}}~~ \text{are disjoint subsets of}~~ F^{m}$ and we define 
		$$\mu_m(g_{w}(F_{w_m}^{m})=P_w~~\forall~~  w=(w_1,w_2,\dots,w_m)\in \Omega_{m}.$$
		At the $2$nd level of $F^{m}$, $\{g_{w}\circ g_{\sigma}(F_{\sigma_m}^{m})\}_{w\in \Omega_{m},\sigma\in \Omega_{m}, w=(i,w_2,\dots,w_m),\sigma=(w_m,\sigma_2,\sigma_3,\ldots,\sigma_m)}$ are disjoint subsets of $F^{m}$ and we define$$\mu_m(g_{w}\circ g_{\sigma}(F_{\sigma_m}^{m})=P_{w *\sigma}~~\forall~~w=(w_1,w_2,\dots,w_m)\in \Omega_m,\sigma=(w_m,\sigma_2,\sigma_3,\ldots,\sigma_m)\in \Omega_m,$$
		where $w*\sigma:=(w_1,w_2,\dots,w_m,\sigma_2,\sigma_3,\ldots,\sigma_m)$. Similarly, we define $\mu_m$ on each level of $F^m$.
		One can easily check that  the measure $\mu_m$ satisfies the following relation
		$$\mu_m\geq \sum_{w\in\Omega_m}p_w\mu_m\circ g_{w}^{-1}.$$
		\par
		Set $\epsilon_0=\min_{w\ne \tau \in \Omega_m}\{d(F_{w}^m,F_{\tau}^m)\}$. Since the RIFS $\mathcal{I}^m$ satisfies the strong separation condition, we have $\epsilon_0>0.$ Let
		$$\delta=\inf\Big\{\Big(\frac{\epsilon_0}{8}\Big)^r \mu_m\Big(B\Big(x,\frac{\epsilon_0}{8}\Big)\Big): x\in F^m \Big\}.$$
		By using the previous lemma, we obtain that $\delta>0.$ Lemma
		\ref{le2.9} yields that $\lim_{n\to \infty}V_{n,r}(\mu_m)=0.$  Thus, there exists an $n_0\in \mathbb{N}$ such that
		$V_{n,r}(\mu)<\delta \te{ for all } n\geq n_0.$ Let $A_n$ be an $n$-optimal set for probability measure $\mu_m$ of order $r$. Then, by Lemma~\ref{le2.11}, we get
		$$\Big(\frac{\|A_n\|_{\infty}}{2}\Big)^r\min\limits_{x\in F^{m}}\mu_m\Big(B\Big(x,\frac{\|A_n\|_{\infty}}{2}\Big)\Big)<\delta .$$
		Since $t\to t^r\min_{x\in F^{m}}\mu_m(B(x,t))$ is an increasing function for $t\in (0,\infty)$, we have $$\|A_n\|_{\infty}<\frac{\epsilon_0}{4} \te{ for all } n\geq n_0.$$
		We define a finite set $A_{n_w}=\{a\in A_n: W(a|A_n)\cap F_{w}^m\ne \emptyset\}$ for $w\in\Omega_m$  and let $n_w$ be the cardinality of $A_{n_w}.$ Since $\|A_n\|_{\infty}<\frac{\epsilon_0}{4} \te{ for all } n\geq n_0$, we obtain that for all $n\geq n_0,$ $A_{n_w}\cap A_{n_\tau}=\emptyset$ for $w\ne \tau\in \Omega_m$ and $\sum_{w\in \Omega_m}n_w\leq n.$   Thus, for $n\geq n_0,$
		\begin{align*}
			V_{n,r}(\mu_m)&=\int d \Big(x,A_{n}\Big)^r d{\mu_m}(x)\\&\geq \sum_{w\in\Omega_m}p_w \int d \Big(g_w(x),A_{n}\Big)^r d{\mu_m}(x)\\&= \sum_{w\in\Omega_m}p_w \int d \Big(g_w(x),A_{n_w}\Big)^r d{\mu_m}(x)\\&\geq \sum_{w\in\Omega_m}p_w C_{w}^{r}\int d \Big(x,g_{w}^{-1}(A_{n_w})\Big)^r d{\mu_m}(x)\\& \geq \sum_{w\in\Omega_m}p_w C_{w}^{r}V_{n_w,r}(\mu_m).
		\end{align*}
		Thus, the proof is complete.
	\end{proof}
	\begin{prop}
		Let $\mathcal{I}^m$ be an RIFS defined as in the previous proposition. Let $r\in (0,\infty).$ Then,
		$$\liminf_{n\to \infty}n e_{n,r}^{l}(\mu_m)>0 \te{ for all } l\in (0,l_{m,r}),$$
		where $l_{m,r}\in (0,\infty)$ is uniquely determined by $\sum_{w\in\Omega_m}(p_w C_{w}^{r})^{\frac{l_{m,r}}{r+l_{m,r}}}=1$. Moreover, $l_{m,r}\leq \underline{D}_r(\mu_m).$
	\end{prop}
	\begin{proof}
		Since $p_w C_{w}^{r}<1$ for each $w\in \Omega_m,$ we have $t\to \sum_{w\in\Omega_m}(p_w C_{w}^{r})^{t}$ is a strictly decreasing continuous function. It is not difficult to see that there exists a unique $l_{m,r}\in (0,\infty)$ such that
		$\sum_{w\in\Omega_m}(p_w C_{w}^{r})^{\frac{l_{m,r}}{r+l_{m,r}}}=1.$ Let $l\in (0,l_{m,r}).$ Then, $\sum_{w\in\Omega_m}(p_w C_{w}^{r})^{\frac{l}{r+l}}>1.$ Let $n_0\in \mathbb{N}$ be the same as in the previous proposition. Let $K_0=\inf\{n^{\frac{r}{l}}V_{n,r}(\mu_m): n<n_0\}$. Clearly $K_0>0.$ We claim that $K_0\leq n^{\frac{r}{l}}V_{n,r}(\mu_m)$ for all $n\in \mathbb{N}.$ We prove the claim by induction on $n\in \mathbb{N}.$ Let $n\geq n_0$ and   $K_0\leq k^{\frac{r}{l}}V_{k,r}(\mu_m)$ for all $k<n.$ By using the previous proposition, we see that there exist $n_w\in \mathbb{N}$ with $\sum_{w\in \Omega_m}n_w\leq n$. Then, we have
		\begin{align*}
			n^{\frac{r}{l}} V_{n,r}(\mu_m)&\geq n^{\frac{r}{l}}\sum_{w\in \Omega_m}p_w C_{w}^{r}V_{n_w,r}(\mu_m)\\&   \geq n^{\frac{r}{l}}\sum_{w\in \Omega_m}p_w C_{w}^{r} K_0 (n_{w})^{\frac{-r}{l}}\\&=K_0\sum_{w\in \Omega_m}p_w C_{w}^{r} \Big(\frac{n_w}{n}\Big)^{\frac{-r}{l}}.
		\end{align*}
		Now, by using the H\"older inequality for negative exponent, we obtain
		$$n^{\frac{r}{l}} V_{n,r}(\mu_m)\geq K_0 \Big(\sum_{w\in \Omega_m}(p_w C_{w}^{r})^{\frac{l}{r+l}}\Big)^{\frac{l+r}{l}}\Big( \sum_{w\in \Omega_m}\Big(\frac{n_w}{n}\Big)^{\frac{-r}{l}\frac{-l}{r}}\Big)^{\frac{-l}{r}}.$$
		Using the fact that $\sum_{w\in\Omega_m}(p_w C_{w}^{r})^{\frac{l}{r+l}}>1$
		and $\sum_{w\in \Omega_m}{n_w}\leq n,$ we get
		$$n^{\frac{r}{l}} V_{n,r}(\mu_m)\geq K_0.$$ Thus, by the principle of induction, we get  $n^{\frac{r}{l}} V_{n,r}(\mu_m)\geq K_0$ holds for all $n\in \mathbb{N}$, which is the claim.
		This implies that
		$$\liminf_{n\to \infty}n e_{n,r}^{l}(\mu_m)\geq K_{0}^{\frac{l}{r}}>0  \te{ for all } l\in (0,l_{m,r}).$$
		Therefore, by Proposition \ref{prop2.12}, we have
		$l_{m,r}\leq \underline{D}_r(\mu_m).$ This completes the proof.
	\end{proof}
	In the upcoming proposition, we determine the lower bound of the quantization dimension of the Borel probability measure supported on the limit set of the RIFS under the strong open set condition.
	\begin{prop}\label{prop3.11}
		Let $\mathcal{I}=\{X; f_{ij}, p_{ij}: 1\leq i,j\leq N\}$ be a bi-Lipschitz recurrent  IFS satisfying the strong open set condition. Let $\mu$ be the Borel probability measure supported on the limit set $E$ of the RIFS $\mathcal{I}$ . Let $r\in (0,\infty)$ and $l_r\in (0,\infty)$ be the unique number such that $\phi_r\Big(\frac{l_r}{r+l_r}\Big)=1.$ Then, we have  $l_r\leq\underline{D}_r(\mu).$
	\end{prop}
	\begin{proof}
		Since the RIFS $\mathcal{I}$ satisfies the strong open set condition, Proposition \ref{SSC} yields that for each $m\in \mathbb{N},$ there exists an RIFS $ \mathcal{I}^m=\{X; g_{w}, p_{ij}: 1\leq i,j\leq N, w \in
		\Omega_{m} \}$
		that satisfies the strong separation condition, where $ g_w =f_w \circ f_{\sigma^i} $ for $ w=(w_1,w_2,\dots,w_{m-1},i)\in \Omega_m $ and $ \sigma^{i}$ is defined as follows: For each $ i \in  \{1,2,\dots,N\}$, there exists a finite sequence $\sigma=(i,\sigma_2,\dots,\sigma_n)\in \Omega^*$ such that $ E_{\sigma}=f_{\sigma}(E_{\sigma_n})\subseteq U_i$ and $ \sigma^*=(\sigma_n,\sigma_{n+1},\dots,i)\in \Omega^*.$ Thus, with the help of $\sigma $ and $\sigma^*$ we define $\sigma^i=(i,\sigma_2,\dots,\sigma_n,\sigma_{n+1},\dots,i).$
		Let $\mu_m$ be the probability measure supported on the limit set $F^{m}$ of the RIFS $\mathcal{I}^m.$ By using Proposition 11.6 in \cite{GL1}, we have $0< {\dim}_{\te H}(\mu)\leq \underline{D}_r(\mu).$ Thus, $0<\underline{D}_r(\mu).$ If $\underline{\dim}_{\te B}(F^{m})\leq \underline{D}_r(\mu)$ holds, then we proceed. Otherwise, we may redefine $\sigma^i$ by choosing larger length of $\sigma^*$ due to irreducibility of $P$ such that $\underline{\dim}_{\te B}(F^{m})\leq \underline{D}_r(\mu)$.  By the previous proposition,  $l_{m,r}\leq\underline{D}_r(\mu_m),$ where $l_{m,r}$ is uniquely determined by $\sum_{w\in\Omega_m}(p_w C_{w}^{r})^{\frac{l_{m,r}}{r+l_{m,r}}}=1.$ Since $\underline{D}_r(\mu_m)\leq \underline{\dim}_{\te B}(F^{m})$ and
		$\underline{\dim}_{\te B}(F^{m})\leq \underline{D}_r(\mu)$, we have
		$\underline{D}_r(\mu_m)\leq \underline{D}_r(\mu).$ Our aim is to show that $l_r\leq\underline{D}_r(\mu).$ We prove it by contradiction. Let $\underline{D}_r(\mu)<l_r.$
		Let $\eta= \min\{c_{\sigma^1}^r,c_{\sigma^2}^r,\dots,c_{\sigma^N}^r\}$ and $\zeta=\max\{p_{ji}c_{ij}^r: 1\leq i,j\leq N\}.$ Using the fact $l_{m,r}\leq \underline{D}_r(\mu)$ and $\sum_{w\in\Omega_m}(p_w C_{w}^{r})^{\frac{l_{m,r}}{r+l_{m,r}}}=1,$  we have
		\begin{align*}
			{\eta}^{\frac{-l_{m,r}}{r+l_{m,r}}}&\geq  \sum_{w\in\Omega_m}(p_w c_{w}^{r})^{\frac{l_{m,r}}{r+l_{m,r}}}\\&\geq \sum_{w\in\Omega_m}(p_w c_{w}^{r})^{\frac{\underline{D}_r(\mu)}{r+\underline{D}_r(\mu)}}\\&=
			\sum_{w\in\Omega_m}(p_w c_{w}^{r})^{\frac{\underline{D}_r(\mu)}{r+\underline{D}_r(\mu)}}(p_w c_{w}^{r})^{\frac{l_r}{r+l_r}}(p_w c_{w}^{r})^{\frac{-l_r}{r+l_r}}\\&\geq \sum_{w\in\Omega_m}(p_w c_{w}^{r})^{\frac{l_r}{r+l_r}}
			{\zeta}^{m\Big({\frac{\underline{D}_r(\mu)}{r+\underline{D}_r(\mu)}}-{\frac{l_r}{r+l_r}}\Big)}.
		\end{align*}
		By Lemma \ref{bb} and $l_{m,r}\leq \underline{D}_r(\mu)<l_r$, we have
		$${\eta}^{\frac{-l_r}{r+l_{r}}}\geq {\zeta}^{m\Big({\frac{\underline{D}_r(\mu)}{r+\underline{D}_r(\mu)}}-{\frac{l_r}{r+l_r}}\Big)} G_{r}^{\frac{-l_r}{r+l_r}},$$
		which is a contradiction for large $m\in \mathbb{N}.$ Thus, $l_r\leq\underline{D}_r(\mu)$. This completes the proof.
	\end{proof}
	\par

	\subsection*{Proof of Theorem~\ref{mainthm}} By Proposition~\ref{prop3.7} and $(1)$ of Proposition~\ref{prop2.12}, we have $\overline D_r(\mu)\leq k_r$. On the other hand, by Proposition~\ref{prop3.11}, we have
	$l_r\leq \underline D_r(\mu)$. Hence, $l_r\leq \underline D_r(\mu)\leq \overline D_r(\mu)\leq k_r$, which completes the proof of the theorem.\\

		\subsection*{ Statements and Declarations:}
	\begin{itemize}
		\item {\bf Ethical approval:} Not applicable 
		\item {\bf Competing Interests:} The authors have no relevant financial or non-financial interests
		to disclose.
		\item {\bf Data availability:} Data sharing is not applicable to this article as no datasets were generated or analysed during the current study.
		\item {\bf Funding:} The authors declare that no funds, grants, or other support were received
		during the preparation of this manuscript.
		\item 	{\bf Author Contributions:} All authors contributed equally in this manuscript.
	\end{itemize}
	
	 \section*{Acknowledgements:} We would like to thank the anonymous referees for their valuable comments and
	suggestions.
	
	\bibliographystyle{amsplain}

\end{document}